\documentclass[a4paper,11pt,reqno]{amsart}
\usepackage{amsfonts,amsmath,amssymb,amsthm,enumerate,xcolor,esint,dsfont,bm,a4}
\usepackage[a4paper,text={5.25in,8.25in},centering]{geometry}
\usepackage{mathrsfs}
\usepackage{cite}
\usepackage{soul}
\usepackage{xcolor}

\usepackage[cp866]{inputenc}
\usepackage[T2A]{fontenc}

\theoremstyle{plain}
\newtheorem{theorem}{\noindent\bf Theorem}[section]
\newtheorem{proposition}[theorem]{\noindent\bf Proposition}
\newtheorem{corollary}[theorem]{\noindent\bf Corollary}
\newtheorem{lemma}[theorem]{\noindent\bf Lemma}
\theoremstyle{remark}
\newtheorem{definition}[theorem]{\noindent\bf Definition}

\newtheorem{remark}[theorem]{\noindent \bf Remark}
\newcommand\norm[1]{\left\lVert#1\right\rVert}
\def\RR{\mathbb{R}}

\DeclareMathOperator{\loc}{loc}

\title[On grand Sobolev spaces and pointwise description of BFS]{On grand Sobolev spaces and pointwise description of Banach function spaces}
\author[P. Jain]{Pankaj Jain} 
\address{ Department of Mathematics\\ 
South Asian University\\
Akbar Bhawan, Chanakya Puri, \\
New Delhi - 110 021, India}
\email{pankaj.jain@sau.ac.in and pankajkrjain@hotmail.com}

\author[A. Molchanova]{Anastasia Molchanova} 
\address{Sobolev Institute of Mathematics\\
Academic Koptyug avenue 4 \\
630090 Novosibirsk, Russia\\ 
\and
Institute of Analysis and Scientific Computing\\
TU Wien\\
Wiedner Hauptstra{\ss}e 8--10\\ 
1040 Wien, Austria}
\email[Corresponding author]{anastasia.molchanova@tuwien.ac.at}

\author[M. Singh]{Monika Singh}
\address{Department of Mathematics\\ 
Lady Shri Ram College For Women
(University of Delhi)\\
Lajpat Nagar, New Delhi - 110 024, India}
\email{monikasingh@lsr.du.ac.in}

\author[S. Vodopyanov]{Sergey Vodopyanov}
\address{Sobolev Institute of Mathematics\\
Academic Koptyug avenue 4 \\
630090 Novosibirsk, Russia}
\email{vodopis@math.nsc.ru}

\date{}

\begin{document}

\begin{abstract}
	We obtain a pointwise description of functions belonging to function spaces with the lattice property.
	In particular, it is valid for Banach function spaces provided that the Hardy--Littlewood maximal operator is bounded.
	We also study weighted grand Sobolev spaces, defined on arbitrary open sets $\Omega$ (of finite or infinite measure) in $\RR^n$, and provide pointwise description of them.
\end{abstract}

\maketitle

\section{Introduction}
\setcounter{equation}{0}

Let  $\Omega \subset \RR^n$ be an open connected set having finite or infinite Lebesgue measure,
which we will denote as 
$|\Omega|$. 
The Lebesgue space $L^q(\Omega)$, $1 \le q < \infty$,
is defined to be the space of all measurable functions 
$f$
on $\Omega$ such that
\vspace{-0.125\baselineskip}
\begin{equation*} \label{eq11}
	\|f\|_{L^q(\Omega)}:=\left( \int\limits_\Omega |f(x)|^q \, dx \right)^{1/q} < \infty,
	\vspace{-0.125\baselineskip}
\end{equation*}
and the Sobolev space $W^{1,q}(\Omega)$ is defined to be the space of all $f \in L^q(\Omega)$ such that the weak gradient%
	\footnote{Recall that $\nabla f=(g_1,g_2,\ldots,g_n)\colon\Omega\to\mathbb R^n$  is a  \textit{weak gradient} of $f\colon\Omega\to\mathbb R$ if  
	$f$, ${\nabla f}\in L^1_{\loc}(\Omega)$ and the identity 
	\vspace{-0.125\baselineskip}
	$$
		\int\limits_\Omega f(x) \frac{\partial \varphi}{\partial x_k}(x)\,dx=
		-\int\limits_\Omega g_k(x)  \varphi(x)\,dx
		\vspace{-0.125\baselineskip}
	$$
	holds for all test functions $\varphi\in C_0^\infty(\Omega)$, $k=1,\ldots, n$.}
$\nabla f$ belongs to $ L^q(\Omega)$ 
with a finite norm
$$
	\|f\|_{W^{1,q}(\Omega)}:=\|f\|_{L^q(\Omega)} + \|\nabla f\|_{L^q(\Omega)}.
$$

Sobolev functions are known to have the following pointwise characterization, which is a classic nowadays.

\begin{theorem}\label{thm:A}
	Let $1< q < \infty$. 
	A function  $f $ belongs to $W^{1,q}(\RR^n)$ if and only if $f \in L^q(\RR^n)$ and there exists a non negative $g \in L^q(\RR^n)$ such that the inequality 
	\begin{equation} \label{pointwise_estimate_sobolev}
		|f(x)-f(y)| \le |x-y|(g(x)+g(y))
	\end{equation}
	holds for all $x$, $y$ outside of some set $\Sigma\subset \RR^n$ of measure zero. 
\end{theorem}

\noindent The necessity part of Theorem~\ref{thm:A} has been firstly obtained in \cite{Boj1991}, while the sufficiency 
has been proved independently in \cite{Haj1996} and \cite{Vod1996}. Afterwards, this description turned out to be useful in various aspects of analysis 
such as Sobolev spaces with higher-order derivatives
\cite{BojHaj1993}, Hardy--Sobolev spaces \cite{KosSak2008}, Sobolev spaces on Carnot groups \cite{Vod1996}, quasiconformal analysis on Carnot groups \cite{Vod2007}, and many others.
Since the formula \eqref{pointwise_estimate_sobolev} does not involve the notion of derivative, 
it became a starting point to define 
a counterpart of Sobolev spaces on metric structures and to investigate its properties \cite{Haj1996}. 
From then on the research in analysis on metric spaces is making steady headway, due to various applications in geometric group theory and nonlinear PDEs.
For comprehensive coverage on this topic the reader is referred to \cite{HeiKosShaTys2015,BonCapHajShaTys2020} and references therein.

On the other hand, Sobolev spaces based on functional classes other than those of Lebesgue functions are of interest for their relevance in a number of fields. 
For instance, there are advantages in using Sobolev--Orlicz spaces for nonlinear elasticity \cite{Bal1977}, Lorentz spaces for the Schr\"{o}dinger equation \cite{BreGal1980} and for the $p$-Laplace system \cite{AlbCiaSbo2017}, grand Sobolev spaces for $p$-harmonic operators \cite{DonSboSch2013,GreIwaSbo1997}, Sobolev spaces with variable exponent for fluid dynamics \cite{DieRuz2003} and many other examples. 
Instead of working with each special case separately, one can 
consider so-called Banach function spaces, 
see \cite{BenSha1988,PicKufJohFuc2013}.
This general approach has recently been very fruitful to extend classical results on different function spaces, for example, the regularity of solutions of PDEs \cite{AlbCiaSbo2017}, the Gagliardo--Nirenberg inequality \cite{FioForRosSou2019} and regularity properties of the inverse mapping \cite{MolRosSou2019}.

In this paper, we ask if a characterization similar to \eqref{pointwise_estimate_sobolev} can be obtained for different function classes. 
For example, a pointwise description in terms of the Young function has been obtained for Sobolev--Orlicz mappings in \cite{Tuo2007}.
The main result of the paper is a pointwise estimate in the spirit of \eqref{pointwise_estimate_sobolev} stated in Theorem~\ref{thm:D} for function spaces with the lattice property provided that the Hardy--Littlewood maximal operator is bounded.
It covers, in particular the concept of rearrangement invariant Banach function spaces, see Corollary~\ref{cor:BFS}, a series of new characterizations of various spaces is also available, see Corollaries~\ref{cor:BFS}--\ref{cor:weighted_Sobolev}. 
Moreover, 
following \cite{SamUma2011,Uma2014}, we define and study the weighted grand Sobolev space $W^{1,q)}_{a}(\Omega,w)$, when $\Omega$ may be unbounded. 
Although these spaces are not Banach function spaces, 
the corresponding description 
is valid, see Corollary~\ref{cor:GS}.

\newpage
\section{The pointwise estimate for Banach spaces with the lattice property}\label{sec:pointwise_estimates}

Let 
$\Omega\subset\RR^n$ be an open set and
$X(\Omega)\subset L^1_{\loc} (\Omega)$ be a Banach space with a norm $\|\cdot\|_{X(\Omega)}$. 
The \textit{homogenious Sobolev space} $L^1X (\Omega)$ 
and \textit{Sobolev space}
$W^1 X (\Omega)$
denote the space of weakly differentiable functions
$f$
with
$\nabla f\in X(\Omega)$
and
$f$, 
$\nabla f\in X(\Omega)$,
respectivily. 
These spaces are equipped with a semi-norm
$\|f\|_{L^1 X (\Omega)}:=\left\|\nabla f\right\|_{X(\Omega)}$
and a norm
$\|f\|_{W^1 X (\Omega)}:=\left\|f\right\|_{X(\Omega)}+\left\|\nabla f\right\|_{X(\Omega)}$.
We say that the norm $\|\cdot\|_{X(\Omega)}$ satisfies the \textit{lattice property} if
for $f$, $g\in X(\Omega)$ such that 
$|f|\leq|g|$ a.e.\ it holds $\|f\|_{X(\Omega)}\leq\|g\|_{X(\Omega)}$.

\begin{definition}
	Let $g\in L^{1}_{\loc}(\mathbb R^n)$ and $0< t \le \infty$.
	We define the {\it Hardy--Littlewood maximal operator} or simply the {\it maximal operator} by
	\vspace{-0.1\baselineskip}
	$$
		M_{t}g(x):=\sup_{0<r\leq t}\fint\limits_{B(x,r)}|g(y)|\,dy,
		\vspace{-0.5\baselineskip}
	$$
	where
	$B(x,r)$ is a ball with radius $r>0$ and center $x \in \mathbb R^n$,
	and
	$\displaystyle\fint_{B} f(x)\,dx$ denotes integral average of $f$ over $B$, i.e.
	\vspace{-0.1\baselineskip}
	\begin{equation}\label{def:average}
		f_B=\fint\limits_{B}f(x)\,dx:=\frac{1}{|B|}\int\limits_{B}f(x)\,dx.
		\vspace{-0.1\baselineskip}
	\end{equation}
	In the case $t=\infty$, we write $Mg(x)$ instead of $M_\infty g(x)$.
\end{definition}

The class of Banach spaces $X(\Omega)\subset L^1_{\loc} (\Omega)$ with the lattice property, s.t.\ the maximal operator $M$ is bounded on $X(\Omega)$,
includes many and various spaces, for example, weighted Lebesgue (with Muckenhaupt's weight), grand Lebesgue, Musielak--Orlicz, Lorentz and Marcinkiewicz spaces, as well as Lebesgue spaces with variable exponents.
In particular, it includes the general concept of Banach function spaces.

The main result of the paper reads as follows.

\begin{theorem}\label{thm:D} 
Let $\Omega\subseteq \RR^n$ be a domain,
$X(\Omega) \subset L^1_{\loc} (\Omega)$ be a Banach space such that 
$X(\Omega)$ satisfies the lattice property
and the maximal operator $M\colon X(\Omega) \to X(\Omega)$ is bounded.
Then  $f \in W^1 X (\Omega)$ if and only if 
$f \in X (\Omega)$ and there exists a non negative function $g\in X(\Omega)$ such that the inequality
\begin{equation}\label{min3}
	|f(x)-f(y)|\leq |x-y| (g(x)+g(y))
\end{equation}
holds for all $x$, $y \in \Omega \setminus S$ with $B(x,3|x-y|)\subset\Omega$,
where $S\subset\Omega$ is a set of measure zero.
\end{theorem}

\begin{remark}
Note, that there are no restrictions on the geometry of $\Omega$.
\end{remark}

Before proving Theorem~\ref{thm:D} we recall the following estimates. 
For the convenience of the reader, we also provide here 
a new proof of Proposition~\ref{pi}.

\begin{lemma}[{\hspace{-.1pt}\cite[4.5.2, Lemma 1]{EvaGar1992}}]\label{lemma1pi} 
There exists a constant $C=C(n)$ such that, 
for any ball $B(\xi,r)\subset \mathbb R^n$ and function $f\in C^1(B(\xi,r))$,
the following inequality
\vspace{-0.1\baselineskip} 
\begin{equation*}
	\int\limits_{B(\xi,r)}|f(y)-f(z)|\,dz\leq Cr^n\int\limits_{B(\xi,r)}\frac{|\nabla f(z)|}{|y-z|^{n-1}}\,dz
	\vspace{-0.1\baselineskip}
\end{equation*}
holds  for any point $y\in B(\xi,r)$.
\end{lemma}

\begin{proposition}[{\hspace{-.1pt}\cite[Lemma 7.16]{GilTru1983}}]\label{pi}
If 
$f\in W^{1,1}_{\loc}(\mathbb R^n)$ then, for every ball $B:=B(\xi,r)\subset \mathbb R^n$ centered at 
$\xi\in\RR^n$ of radius $r$,
the inequality
\begin{equation} \label{41}
	|f(x)-f_B|\leq C\int\limits_{B}
	\frac{|\nabla f(y)|}{|x-y|^{n-1}}\,dy
\end{equation}
holds for  every $x\in B\setminus S_B$, where  $S_B$ is measurable subset of $B$ of zero Lebesgue measure:  
$|S_B|=0$,
and $f_B$ is the integral average 
defined by \eqref{def:average}.
\end{proposition}

\begin{proof}

By Lemma \ref{lemma1pi}, for a ball  $B:=B(\xi,r)\subset \mathbb R^n$,  $f\in C^1(B)$,  $x\in B$ and $\rho$
small enough ($B(x,\rho)\subset B$),
we come to the following inequalities
\begin{multline*}
\begin{aligned}
	\fint\limits_{B(x,\rho)}|f(y)-f_B|\,dy & =
	\fint\limits_{B(x,\rho)}\bigg|\fint\limits_{B} (f(y)-f(z))\,dz\bigg|\,dy
	\\
	& \leq
	\fint\limits_{B(x,\rho)}\fint\limits_{B} |f(y)-f(z)|\,dz\,dy \\
	& \leq
	C r^n\fint\limits_{B(x,\rho)}\fint\limits_{B} \frac{|\nabla f(z)|}{|y-z|^{n-1}}\,dz\,dy.
\end{aligned}
\end{multline*}
Now the Lebesgue Differentiation Theorem, with $\rho\to 0$, gives just \eqref{41}
for $C^1$-functions.

In view of the estimate
\begin{equation*} \label{42}
	\int\limits_{B}|f(x)-f_B|\,dx\leq 
	C\int\limits_{B}\int\limits_{B}\frac{|\nabla f(y)|}{|x-y|^{n-1}}\,dy\,dx\leq C_1r\int\limits_{B}|\nabla f(y)|\,dy
\end{equation*}
we can apply standard approximation arguments and to justify 
 inequality \eqref{41} for any $f\in W^{1,1}_{\loc}(\mathbb R^n)$.  
\end{proof}

\begin{proof}[Proof of Theorem~\ref{thm:D}]
Let $\Omega = \RR^n$ (the case $\Omega \subset \RR^n$ is considered in the same way) and
$f\in L^1 X(\mathbb R^n) \cap L^1_{\loc} (\mathbb R^n)$ and note that \eqref{41} holds for $f$. 
Now, for
$x$, $y\in \mathbb R^n$, consider a~ball $B$ of the least radius
with center $\xi=\frac12(x+y)$ such that $x$, $y\in\overline B$
and
$|x-\xi|=|y-\xi|$.
Applying \eqref{41}, we obtain
\begin{equation}\label{t1}
\begin{aligned}
	|f(x)-f(y)| &\leq |f(x)-f_B|+|f(y)-f_B|\\
	& \leq C \left(\,\int\limits_{2B}
	\frac{|\nabla f(z)|}{|x-z|^{n-1}}\,dz +
	\int\limits_{2B}
	\frac{|\nabla f(z)|}{|y-z|^{n-1}}\,dz\right)\\
	 &:= C(I_1+I_2)
\end{aligned}
\end{equation}
for all $x$, $y\in B\setminus S_B$ where $|S_B|=0$.

It is obvious that
$2B \subset B_1=B(x, t)$
for
$t=3|x-\xi|$.
Thus, we get the estimate,
which first appeared in \cite{Hed1972},
\begin{equation*}\label{t2}
\begin{aligned}
	I_1 &\leq \int\limits_{B_1}
	\frac{|\nabla f(z)|}{|x-z|^{n-1}}\,dz \\
	& =
	\sum _{j=1}^{\infty}
	\int\limits_{\frac{1}{2^{j-1}}B_1\setminus\frac{1}{2^j}B_1}
	\frac{|\nabla f(z)|}{|x-z|^{n-1}}\,dz \nonumber\\
	&\leq \omega_nt\sum_{j=1}^{\infty}\frac{2^n}{2^j}\frac{1}{\bigl|\frac{1}{2^{j-1}}B_1\bigr|}
	\int\limits_{\frac{1}{2^{j-1}}B_1}
	|\nabla f(z)|\,dz  \nonumber\\
	&\leq 2^{n}\omega_nt M_{t}(|\nabla f|)(x)
\end{aligned}
\end{equation*}
where $\omega_n$ is the volume of the unit ball.
Similarly, $I_2$ can be estimated. Using these estimates,  
\eqref{t1} gives
\begin{equation}\label{t3}
	|f(x)-f(y)| \leq 3C\cdot 2^{n-1}\omega_n|x-y|
	(M_{t}(|\nabla f|)(x)
	+M_{t}(|\nabla f|)(y))
\end{equation}
for all $x$, $y\in B\setminus S_B$ with $|S_B|=0$. 
Now, take $g= 3C\cdot2^{n-1}\omega_n M(|\nabla f|)$.
Since $|\nabla f| \in X(\mathbb R^n)$ 
and the maximal operator $M \colon X(\mathbb R^n) \to X(\mathbb R^n)$ is bounded,
we conclude that  $g \in X(\mathbb R^n)$. 
Hence, by \eqref{t3} we obtain~\eqref{min3}.

Conversely, 
 the inequality  \eqref{min3} holds
with  some functions 
$f\in L^1_{\loc}(\mathbb R^n)$
 and  
 $g\in X(\mathbb R^n)$ 
for almost all  $x$, $y \in \mathbb R^n\setminus S$ where  $S$ is some set of zero measure. 
Under hypothesis of theorem
we have $g\in L^1_{\loc}(\mathbb R^n)$. 

For any   $k\in\mathbb N$ we define a set 
$$ 
A_k=\{x \in \mathbb R^n\setminus S: g(x)\leq k \}.
$$
Then for all points  $x$, $y\in  A_k$ we have 
\begin{equation*}
	|f(x)-f(y)|\leq2k|x-y|.
\end{equation*}
Therefore  $f$ is a Lipschitz function
on the set $A_k$
in the conventional sense.
Applying the Kirszbraun extension theorem \cite[3.1, Theorem 1]{EvaGar1992},
we obtain an~extension $\tilde f_k\colon \mathbb R^n \to \mathbb R$ of~
$f\colon A_k\to \mathbb R$ to a Lipschitz function on 
$\mathbb R^n$ with the same Lipschitz constant. In particular,
for all points  $x$, $y\in  \mathbb R^n\setminus A_k$ 
we have
 \begin{equation}\label{Lipfun}
	|\tilde f(x)-\tilde f(y)|\leq2k|x-y|.
\end{equation}

Take an arbitrary vector  $e_i$ of the standard basis in  $\mathbb R^n$, $i=1,2,...,n$.
We will write points  $x\in \mathbb R^n$  as 
$x=(\bar x, x_i)=\bar x+ x_ie_i$  where  $\bar x\in \mathbb R^{n-1}$ and $x_i\in \mathbb R$.
 For any 
 $\bar x\in \mathbb R^{n-1}$ the restriction  
$\mathbb R\ni x_i\mapsto \tilde f_k(\bar x+x_i e_i) $ is a Lipschitz function with respect to  $x_i\in \mathbb R$. 
Hence it is absolutely continuous on every line parallel to $i$-coordinate axis and therefore 
it has the partial derivative~
$\frac{d\tilde f_k}{dx_i}(\bar x+x_ie_ i)$ for almost all $x_i\in \mathbb R$.  Thus, by Tonelli's theorem (see, for example, \cite[Theorem~13.8]{Sch2005})
the partial derivative $\frac{d\tilde f_k}{dx_i}(x)$ exists for almost all
 $x\in \mathbb R^n$, $i=1,2,...,n$. 

Again by Tonelli's theorem the intersection  
$$
\{\bar x+x_ie_i: x_i\in \mathbb R\}\cap A_k
$$
is measurable for almost all $\bar x\in\mathbb R^{n-1}$. 
Take  $\bar x\in \mathbb R^{n-1}$ such that  

\begin{enumerate}
\item[(1)] the intersection 
$\{\bar x+x_ie_i: x_i\in \mathbb R\}\cap A_k$  is measurable and has a positive Lebesgue measure;

\item[(2)] the restriction of $g$  to the line $\{\bar x+x_ie_i: x_i\in \mathbb R\}$ belongs to the class $L^1_{\loc}$.
\end{enumerate}

Now let  $t\in \mathbb R$ be a value such that

\begin{enumerate}
\item[(3)] $\bar x+te_i\in A_k$;

\item[(4)] $\bar x+te_i$ is a density point with respect to  $\{\bar x+x_ie_i: x_i\in \mathbb R\}\cap A_k$;

\item[(5)] there exists the partial derivative  
$$
\frac{\partial }{\partial x_i}\tilde f(\bar x+te_i) =\frac{d}{dt}\tilde f(\bar x+te_i);
$$ 

\item[(6)]  $\bar x+te_i$ is a Lebesgue point of the restriction $g\colon\{\bar x+x_ie_i: x_i\in \mathbb R\}\to \mathbb R$.
\end{enumerate}

Properties (1)--(2) are fulfilled for almost all $\bar x \in \mathbb R^{n-1}$ such 
that the intersection $\{\bar x + x_ie_i: x_i \in \mathbb R\} \cap A_k $
is not empty.
For fixed $\bar x \in \mathbb R^{n-1} $ properties (3)--(6) hold for almost
all $t$ such that $ \bar x + te_i \in A_k $.
Therefore, by Tonelli's theorem, properties (1)--(6) are satisfied for almost all $x \in A_k$.

Our immediate goal is to evaluate the derivative: 
\begin{equation}\label{min9}
\biggl|\frac{\partial\tilde f_k}{\partial x_i} \biggr|(x)\leq
\begin{cases}
 2g(x) \quad &\text{for almost all} \quad  x \in A_k, \ \ 
	i=1,2,\ldots,n,\\
	2k \quad &\text{for almost all} \quad  x \in \mathbb R^n\setminus A_k.
\end{cases}
\end{equation}
The second line is a consequence of \eqref{Lipfun} and the definition of partial derivative.

Let $x = \bar x + te_i \in A_k$ be  a point meeting  all the above-mentioned 
properties (1)--(6). In view of  \eqref{min3} for the function
$$
\mathbb R\ni\tau\to  h(\tau)=\tilde f(\bar x+ (t+\tau )e_i)=\tilde f(x+ \tau e_i)
$$ 
we have an estimate of the difference relation provided
$x+ \tau e_i\in A_k$:
\begin{equation}\label{min10}
\begin{aligned}
	\biggl|\frac{h(\tau)-h(0)}{\tau}\biggr|
	&=\biggl|\frac{\tilde f(x+ \tau e_i)-\tilde f(x)}{\tau}\biggr| \\
	&=\biggl|\frac{f(x+\tau e_i)-f( x)}{\tau}\biggr| \\
	&\leq g(x)+g(x+\tau e_i)\\
	&=2g(x)+g(x+\tau e_i)-g(x).
\end{aligned}
\end{equation}
  The last row of the relations \eqref{min10} shows that the estimate for the derivative
  $ h '(0) $ depends on the behavior of the difference
  $g(x+\tau e_i)-g(x)$. 
Since $x$ is the Lebesgue point of the restriction
$g\colon\{\bar x+x_ie_i: x_i\in \mathbb R\}\to \mathbb R$
 then
$$
\frac{1}{2\delta}\int\limits_{-\delta}^\delta|g(x+\tau e_i)-g( x)|\, d\tau = o(1)\quad\text{as  $\delta\to 0$}.
$$
We fix an arbitrary number $\varepsilon> 0$.
Using the Chebyshev inequality, we deduce
\begin{equation*}
	|\{\tau\in(-\delta,\delta): |g(x+\tau e_i)-g(x)|\geq \varepsilon\}|
	\leq
	\frac1\varepsilon\int\limits_{-\delta}^\delta|g(x+\tau e_i)-g(x)|\, d\tau = 
	\frac{o(\delta)}\varepsilon
\end{equation*}
as $\delta\to 0$.
Therefore, we obtain
\begin{equation}\label{min11}
	1\geq\frac{|\{\tau\in[-\delta,\delta]: |g(x+ \tau e_i)-g(x)|<\varepsilon\}|}
	{2\delta}
	\geq\frac{2\delta-\frac{o(\delta)}\varepsilon}{2\delta}\to1
\end{equation}
as $\delta\to 0$.

As soon as the point $x$ is a density point  with respect to the intersection 
$\{\bar x + x_ie_i: x_i \in \mathbb R \} \cap A_k$, we come to 
\begin{equation}\label{min12}
	\frac{|\{[x-\delta e_i,x+\delta e_i]\cap A_k\}|}{2\delta}
	\to1\quad\text{as $\delta\to 0$}.
\end{equation}

We introduce the notation
\begin{equation*}
T=\{\tau\in[-1,1]: |g( x+ \tau e_i)-g(x)|<\varepsilon\}.
\end{equation*}
By virtue of \eqref{min11}, the point $0$ is a density point with respect to the set $T$.
Similarly to the previous one, due to \eqref{min12}, 
the point $0$ is also a density point
with respect to the set
$$
	P=\{\tau\in[-1,1]: [x-\tau e_i, x+ \tau e_i]\cap A_k\}.
$$
From the definition of a density point, we conclude that $0$
is the density point of the  intersection 
$T \cap P$.
From here, \eqref{min11} and \eqref{min12}
we derive the relations
\begin{equation}\label{min13}
 \begin{aligned}
   \biggl|\frac{h(\tau)-h(0)}{\tau}\biggr|
   &=\biggl|\frac{\tilde f(x+ \tau e_i)-\tilde f(x)}{\tau}\biggr|
   \\
   &=
   2g(x)+g(x+\tau e_i)-g(x)
   \\
   &\leq
   2g(x)+\sup\limits_{\tau\in [-\delta,\delta]\cap(T \cap P)}|g(x+\tau e_i)-g( x)|
   \\
   &\leq
   2g(x)+ \varepsilon
  \end{aligned}
 \end{equation}
   for all points $ \tau \in [- \delta, \delta] \cap (T \cap P)$. 
   Passing to the limit in \eqref{min13} as $\tau \to0 $, $ \tau \in [-\delta, \delta] \cap (T \cap P)$,	
   we get 
   $$
  |h'(0)| = \biggl|\frac{\partial \tilde f_k}{\partial x_i}(x)\biggr|\leq
 2g(x)+ \varepsilon.
$$
Since here $ \varepsilon> 0$ is an arbitrary positive number, the inequality \eqref{min9}
is proved.

As long  as $ A_k \subset A_{k + 1} $, $ k = 1,2, \ldots$, and
$\Bigl|\mathbb R^n\setminus\bigcup \limits_{k=1}^{\infty}A_k\Bigr|=0$,
for almost all $ x \in \mathbb R ^ n $ there exist limits  
$$
\lim\limits_{k\to\infty} \tilde f_k(x)=f(x),
$$
and, for 
$i=1,2,\ldots,n$,
by \eqref{min10},
we define
$$
w_i(x):=\lim\limits_{k\to\infty}\frac{\partial \tilde f_k}{\partial x_i}(x)=
\begin{cases}
	\frac{\partial \tilde f_1}{\partial x_i}(x),
	\quad \text{if   $x\in A_1$},\\
	\frac{\partial \tilde f_l}{\partial x_i}(x),
	\quad \text{if  $x\in A_l\setminus A_{l-1}$,} \quad  l=2,3,\ldots.
\end{cases}
$$
Fixing  an arbitrary point $x_0\in A_k$ and taking into account the inequality $k\leq g(x)$ at $x\in \RR^n\setminus A_{k}$
we have the following estimates:  
$$|\tilde f_k(x)-\tilde f_k(x_0)|\leq 2k|x-x_0|\leq 2g(x)|x-x_0|$$  
at
$x\in \RR^n\setminus (A_{k}\cup S)$. 
Therefore 
\begin{equation}\label{min14}
	|\tilde f_k(x)|\leq
	\begin{cases}
		|f(x)|,
		\ &\text{if   $x\in A_k$},\\
		2g(x)|x-x_0|+|f(x_0)|,
		\  &\text{if   $x\in \mathbb R^n\setminus  (A_{k}\cup S)$}.
	\end{cases}
\end{equation}
Moreover
$|w_i(x)|\leq 2g(x)$ for almost all  $x\in \mathbb R^n$ by \eqref{min9} and the inequality
$k\leq g(x)$ in $x\in \mathbb R^n\setminus  A_{k}$.

Take an arbitrary test function $\varphi \in C_0^\infty(\RR^n)$.
Under hypotheses of theorem, the function $f$ is integrable  on
$ \operatorname{supp} \varphi $.
Since $ \tilde f_k $ has the first generalized derivatives, we have
$$
\int\limits_{\{x:\varphi(x)\ne0\}}\varphi(x)\frac{\partial \tilde f_k}{\partial x_i}(x)\,dx=
-\int\limits_{\{x:\varphi(x)\ne0\}}\tilde f_k(x)\frac{\partial\varphi}{\partial x_i}(x)\,dx.
$$
On the compact set $\{x:\varphi(x)\ne0\}$ the sequences $\tilde f_k(x)$ and $\frac{\partial \tilde f_k}{\partial x_i}(x)$ have the majorants 
$\max\limits_{\{x:\varphi(x)\ne0\}}(|f(x)|,2g(x)|x-x_0|+|f(x_0)|)$
 and $2g(x)$, respectively.
Therefore, by the Lebesgue dominated convergence theorem
we obtain
$$
\int\limits_{\{x:\varphi(x)\ne0\}}\varphi(x) w_i(x)\,dx=
-\int\limits_{\{x:\varphi(x)\ne0\}} f(x)\frac{\partial\varphi}{\partial x_i}(x)\,dx
$$
for all  $\varphi\in C_0^\infty(\RR^n)$.
Consequently,
the weak derivative
$w_i(x)$ 
equals the partial derivative
$ \frac{\partial f} {\partial x_i} (x)$, $ i = 1, \ldots, n$,
a.e.\ in $\RR^n$.
By virtue of \eqref{min9} and the lattice property, we have the estimate
$$
\norm{ \nabla f}_{X(\mathbb R^n)}\leq 2 \sqrt{n}\, \norm{g}_{X(\mathbb R^n)}.
$$
Therefore, 
$\nabla f\in X(\mathbb R^n)$.
Thus, it is proved that 
$f\in L^1 X(\mathbb R^n)$.
\end{proof}

Let us provide the following characterization of $L^1X(\Omega)$, that generalizes \cite[Theorem 1]{Haj1996}.

\begin{corollary}\label{cor:PD_W}
Let $\Omega\subseteq \RR^n$ be a domain, $X(\Omega) \subset L^1_{\loc} (\Omega)$ be a Banach space such that 
$X(\Omega)$ satisfies the lattice property
and the maximal operator $M\colon X(\Omega) \to X(\Omega)$ is bounded.
Then $f \in L^1 X (\Omega)$ if and only if there exists a non negative function $g\in X(\Omega)$ such that the inequality
\begin{equation*}
	|f(x)-f(y)|\leq |x-y| (g(x)+g(y))
\end{equation*}
holds for all $x$, $y \in \Omega \setminus S$ with $B(x,3|x-y|)\subset\Omega$,
where $S\subset\Omega$ is a set of measure zero.
\end{corollary}

\begin{remark}
The proof of Theorem \ref{thm:D} is based
on arguments of the paper \cite{Vod1996} where the similar assertion for Sobolev functions, defined on Carnot groups, was given. 
The provided details were given during lectures it the Mathematical Department of Novosibirsk State University without the arguments with approximate derivative
from \cite{Vod1996}.
See also \cite{Haj1996} for an independent  proof for Sobolev functions defined on Euclidean spaces.
\end{remark}

\section{The pointwise estimate for Banach function spaces}\label{sec:pointwise_estimates_BFS}
\setcounter{equation}{0}

Let us provide the following corollaries of Theorem~\ref{thm:D}
when the condition on the maximal operator can be explicitly stated.

\begin{corollary}\label{cor:BFS}
	Let $X(\RR^n)$ be a rearrangement invariant Banach function space with
 	the upper Boyd index $\beta_X<1$.
	A function $f$ belongs to $W^1X(\RR^n)$ 
	if and only if 
	$f\in X(\RR^n)$ and 
	there exists a non negative function $g\in X(\RR^n)$ 
	such that the inequality
	\begin{equation}\label{thmE:characterization}
		|f(x)-f(y)|\leq |x-y|(g(x)+g(y))
	\end{equation}
	holds for 
	$x$, $y \in \RR^n$ 
	a.e.
\end{corollary}

We refer the reader to \cite{BenSha1988} and \cite{PicKufJohFuc2013} for definitions and properties of different Banach function spaces.

\begin{corollary}\label{cor:BFS}
	Let $X(\RR^n)$ be a rearrangement invariant Banach function space with
 	the upper Boyd index $\beta_X<1$.
	A function $f$ belongs to $W^1X(\RR^n)$ 
	if and only if 
	$f\in X(\RR^n)$ and 
	there exists a non negative function $g\in X(\RR^n)$ 
	such that the inequality
	\begin{equation}\label{thmE:characterization}
		|f(x)-f(y)|\leq |x-y|(g(x)+g(y))
	\end{equation}
	holds for 
	$x$, $y \in \RR^n$ 
	a.e.
\end{corollary}

\begin{remark}\label{rem:boyd}
	If
	$X$ is a rearrangement invariant Banach function space then the maximal operator is bounded if and only if
 	the upper Boyd index $\beta_X<1$,
    see \cite[Chapter 3, Definition 5.12 and Theorem 5.17]{BenSha1988}.
    The formulas for calculating the Boyd indices of classical function spaces may be found in, for example, \cite{FioKrb1998}.
\end{remark}

For reader's convenience, we provide below explicit formulation for  Lorentz and Orlicz spaces, as well as Lebesgue spaces with variable exponent, since these spaces are of special interest in applications.

\begin{corollary}\label{cor:Lorentz}
	A function $f$ belongs to Sobolev--Lorentz space $W^1L^{p,q}(\RR^n)$ 
	with $p>1$ and $q \geq 1$
	if and only if 
	$f\in L^{p,q}(\RR^n)$ and 
	there exist a non negative function $g\in L^{p,q}(\RR^n)$ 
	such that the inequality~\eqref{thmE:characterization} is fulfilled 
	for $x$, $y \in \RR^n$ a.e.
\end{corollary}
\noindent Properties of Lorentz spaces can be found in \cite[Chapter 4.4]{BenSha1988}.

\begin{corollary}\label{cor:Orlicz}
	Consider a Young function $A$ such that
	there exists a positive constant $c$, for which
    \begin{equation*}
		\int\limits_0^t\frac{A(s)}{s^2}\,ds\leq\frac{A(ct)}{t}
 	\end{equation*}
	holds for all $t>0$.
	Then $f$ belongs to Sobolev--Orlicz space $W^1L^{A}(\RR^n)$ 
	if and only if 
	$f\in L^{A}(\RR^n)$ and 
	there exists a non negative function $g\in L^{A}(\RR^n)$ 
	such that the inequality~\eqref{thmE:characterization} is fulfilled 
	for $x$, $y \in \RR^n$ a.e.
\end{corollary}
\noindent For the research on the maximal operator in Orlicz spaces the reader is referred to \cite{Kit1997} and \cite[Theorem 3.3]{Mus2019}.

\begin{remark}
	In \cite[Theorem 1.2]{Tuo2007}
	the characterization 
	in spirit of \cite{Haj2003} is given 
	for the Orlicz--Sobolev space $W^1L^{A}(\RR^n)$ with the Young function $A$ and a complimentary $A^*$ satisfying the $\Delta_2$-condition.
	The right hand side of \eqref{thmE:characterization} is expressed in terms of the Young function:
	\begin{equation*}
		|f(x)-f(y)|\leq C |x-y|(A^{-1}(M_{\sigma|x-y|}A(g)(x))+A^{-1}(M_{\sigma|x-y|}A(g)(y))),
	\end{equation*}
	for some $C>0$ and $\sigma \geq 1$.
	Here $A^{-1}$ is a generalized inverse of $A$.
\end{remark}

\begin{corollary}\label{cor:variable_exponent}
	Let $p\colon \RR^n \to [1,\infty]$ be a measurable function with $p^{-}: = \operatorname{ess\,inf}_{y\in\RR^n}p(y)>1$ and $\frac{1}{p}$ is globally log-H\"older continuous%
	\footnote{
		$\alpha\colon\RR^n\to\RR$ is globally log-H\"older continuous if there exist $c_1$, $c_2>0$ and $\alpha_{\infty} \in \RR$ s.t.\
		$$
			|\alpha(x) - \alpha(y)| \leq \frac{c_1}{\operatorname{log}\left(e + \frac{1}{|x-y|}\right)} \text{ for all } x, y \in \RR^n, \: \text{ and }\: 
			|\alpha(x) - \alpha_\infty| \leq \frac{c_2}{\operatorname{log}(e + |x|)} \text{ for all } x \in \RR^n.
		$$
	}.
	Then $f$ belongs to Sobolev space with variable exponent $W^{1,p(\cdot)}(\RR^n)$ 
	if and only if 
	$f\in L^{p(\cdot)}(\RR^n)$ and 
	there exist a non negative function $g\in L^{p(\cdot)}(\RR^n)$ 
	such that the inequality~\eqref{thmE:characterization} 
	holds for 
	$x$, $y \in \RR^n$ 
	a.e.
\end{corollary}
\noindent The theory of Lebesgue and Sobolev spaces with variable exponent and, in particular, conditions about the boundedness of maximal operator can be found in \cite[Section 4.3]{DieHarHasRuz2017}).


\section{Weighted grand Sobolev spaces}\label{sec:weighted_GS}
\setcounter{equation}{0}

In this section, we study weighted grand Lebesgue $L^{q)}_{a}(\Omega,w)$ and grand Sobolev spaces $W^{1,q)}_{a}(\Omega,w)$, defined on a bounded or unbounded domain $\Omega$.
In case of a bounded domain $\Omega$ a grand Lebesgue space $L^{q)}(\Omega)$ is a Banach function space (Proposition~\ref{prop:generalizedGL}) and the maximal operator is bounded (Proposition~\ref{HLGL}). 
The case of an unbounded domain, instead, does not fit the concept of Banach function spaces. 
Nevertheless, the maximal operator is bounded in $L^{q)}_{a}(\mathbb R^n,w)$ 
(Theorem~\ref{thm:E}) and, therefore, the pointwise description is possible (Corollary~\ref{cor:GS}).

\subsection{Grand Lebesgue spaces}\label{sec:grand_lebesgue}
\setcounter{equation}{0}

Let us start with a definition of grand Lebesgue spaces.

\begin{definition}
Let $\Omega \subset \mathbb R^n$ be a bounded domain for $n \geq 2$. 
For $1<q <\infty $, the {\it grand Lebesgue space} $L^{q)}(\Omega)$ consists of all measurable functions $g$,
$g \in \bigcap\limits_{1< p<q} L^p(\Omega)$, 
such that
\begin{equation}\label{glsbd}
	\|g\|_{L^{q)}(\Omega)}:=\sup\limits_{0<\varepsilon< q-1}
	\Biggl(\varepsilon \int\limits_\Omega
	|g(x)|^{q-\varepsilon}\,dx\Biggr)^\frac1{q-\varepsilon}
	< \infty.
\end{equation}
\end{definition}

\begin{remark}
	In some cases it is more convenient to consider a multiplier $\frac{1}{|\Omega|}$ in front of the integral in \eqref{glsbd}. 
\end{remark}

\noindent The spaces $L^{q)}(\Omega)$ are rearrangement invariant Banach function spaces and the following continuous embeddings hold:
$$L^q(\Omega) \hookrightarrow L^{q)}(\Omega) \hookrightarrow L^{q-\varepsilon}(\Omega) \, \, \,\text{for} ~0<\varepsilon < q-1. $$
For any given $1<q<\infty$, the inclusion $L^q(\Omega) \subset L^{q)}(\Omega)$ is strict and, moreover, the spaces $L^{q)}(\Omega)$ are not reflexive \cite{Fio2000}.
A grand Lebesgue space $L^{n)}(\Omega)$, where $\Omega$ is a bounded domain in $\RR^n$, was first introduced in \cite{IwaSbo1992}, to study the question of the integrability of the Jacobian of an orientation-preserving mapping belonging to the Sobolev space 
$W^{1,n}(\Omega)$.  
Grand Lebesgue spaces have been thoroughly studied in the one-dimensional case when $\Omega =I= (0,1)$, 
some recent results can be found in \cite{AnaFio2015,CapForGio2013,Fio2000,FioGupJai2008,JaiSinSin2017-1,JaiSinSin2017-2}.

To define and deal with grand Lebesgue space $L^{q)}(\Omega)$ in case of unbounded domains $\Omega \subset \mathbb R^n$, i.e.\  $|\Omega| \le \infty$, one needs to introduce a class of weights.

\begin{definition}
	A \textit{weight} $w$, i.e.\ a measurable, positive, and finite almost everywhere (\textit{a.e.}) function, is said to be in the {\it Muckenhoupt class}  $A_q(\mathbb R^n )$, $1 < q< \infty$, if 
	$$[w]_{A_q(\mathbb R^n )}:= \sup_{Q\subset\RR^n}\left( \fint\limits_Q w(x) dx \right) \left( \fint\limits_Q w(x)^{-\frac1{q-1}}\,dx\right)^{q-1} <\infty, $$
	where $Q \subset \mathbb R^n$ is a cube with edges parallel to the coordinate axes.
\end{definition}

The class $A_q(\mathbb R^n ) $, $1 < q< \infty$, possesses the following properties which are consequences of the H\"older inequality and the reverse H\"older inequality, see for instance \cite[Theorem IV]{CoiFef1974}.

\begin{proposition}\label{Aq_prop}
	Let $1 < q< \infty$.
	\begin{itemize}
	\item [(i)] If $q < p$, then $A_q(\mathbb R^n ) \subseteq A_p(\mathbb R^n ) $ and $[w]_{A_p(\mathbb R^n )} \le [w]_{A_q(\mathbb R^n )}$$;$

	\item [(ii)] $[w]_{A_q(\mathbb R^n )} \ge 1$$;$

	\item [(iii)]  If $w \in A_q(\mathbb R^n )$ and $0 \leq \alpha \leq 1$, then $w^{\alpha} \in A_{q}(\mathbb R^n )$ and $[w^{\alpha}]_{A_q(\mathbb R^n )} \le [w]^{\alpha}_{A_q(\mathbb R^n )}$$;$

	\item [(iv)] If $w \in A_q(\mathbb R^n )$, then there exists $0 < \sigma < q-1$ such that $w \in A_{q- \sigma} (\mathbb R^n )$$;$

	\item [(v)] If $w \in A_q(\mathbb R^n )$, then there exists $\alpha > 1$ such that $w^{\alpha} \in A_{q}(\mathbb R^n )$.
	\end{itemize}
\end{proposition}

Recall that, if $w$ is a weight, the \textit{weighted Lebesgue space} 
$L^q(\Omega,w)$ and \textit{weighted Sobolev space} $W^{1,q}(\Omega,w)$ are defined with norms 
$$
	\|f\|_{L^q(\Omega,w)}:=\left( \int\limits_\Omega |f(x)|^q w(x) \, dx \right)^{1/q} 
	< \infty
$$
and
$$
	\|f\|_{W^{1,q}(\Omega,w)}:=\|f\|_{L^q(\Omega,w)} + \|\nabla f\|_{L^q(\Omega,w)} < \infty.
$$
respectively.

Now we are ready to define a generalized version of grand Lebesgue spaces.

\begin{definition}\label{def:generalizedGL}
	Let
	$1<q<\infty$
	and
	$w$, $a$ 
	be weights such that
	$ wa^\varepsilon \in L^{1}_{\loc}(\Omega)$, 
	for all 
	$\varepsilon \in (0, q -1)$.
	The {\it generalized grand Lebesgue space} 
	$L^{q)}_{a}(\Omega,w)$
	consists of all measurable 
	functions 
	$g$ 
	defined on 
	$\Omega$ 
	such that
	$$ \|g\|_{L^{q)}_{a}(\Omega,w)} :=
	\displaystyle \sup _{0<\varepsilon< q-1}
	\left( \varepsilon \int\limits_\Omega |g(x)|^{q-\varepsilon}w(x)a^\varepsilon(x)\,dx \right)^\frac{1}{q-\varepsilon}
	< \infty. $$ 
\end{definition}

\begin{remark}
	The space $L^{q)}(\Omega)$ over the unbounded domain $\Omega$ was introduced and developed
	in \cite{SamUma2011,Uma2014}.
	The weight function $a$ in the Definition~\ref{def:generalizedGL} of grand Lebesgue space
	$L^{q)}_{a}(\Omega,w)$
	is called the \textit{grandisator}. 
	It is introduced for the close control of the behavior of 
	$g\in L^{q)}_{a}(\Omega,w)$ at infinity.
\end{remark}

To work with weights in the context of grand spaces the following lemma (see \cite[Lemma 5]{Uma2014}) appears to be useful.

\begin{lemma}\label{lemma5}
	If $w \in A_q(\mathbb R^n )$ and $a^\delta \in A_q(\mathbb R^n)$ for some $\delta > 0$, then there exists $0< \varepsilon < \delta$ such that $w a^{\varepsilon} \in A_{q-\varepsilon}(\mathbb R^n )$.
\end{lemma}

\noindent Note that if $\delta \geq q-1$, then it may happen that $q-\varepsilon <1$. 
In this case one should consider $\alpha < \frac{q-1}{\delta} < 1$ and $\delta_0 = \alpha \delta < q-1$. Then by Proposition~\ref{Aq_prop}(iii) we can apply Lemma~\ref{lemma5} to $a^{\delta_0} \in A_q(\RR^n)$.

It has been proved in \cite[Lemma 3]{Uma2014} that  the following chain of embeddings holds 
\begin{equation} \label{12}
	L^q(\Omega,w) \hookrightarrow L^{q)}_{a}(\Omega,w)
	\hookrightarrow L^{q-\varepsilon}(\Omega,wa^\varepsilon), \quad
	0 < \varepsilon < q- 1,
\end{equation}
if and only if $a\in L^q(\Omega,w)$. 

Some of the properties possessed by the space $L^{q)}_{a} (\Omega,w)$ are proved below:

\begin{proposition} \label{prop:generalizedGL}
	Let 
	$ w a^\varepsilon \in L^{1}_{\loc} (\Omega)$ for all $\varepsilon \in (0, q-1)$.
	Then
	\begin{itemize}

		\item[$($i$)$] if $a \in L^q(\Omega,w)$, then the generalized grand Lebesgue space  $L^{q)}_{a} (\Omega,w)$ is a Banach space$;$

		\item[$($ii$)$]  if $|f| \le |g|$ a.e.\ on $\Omega$ then $\|f\|_{L^{q)}_{a}(\Omega,w)} \le \|g\|_{L^{q)}_{a}(\Omega,w)}$$;$

		\item[$($iii$)$]  if $0\le f_n \nearrow f$ a.e.\ in $\Omega$ then $\|f_n\|_{L^{q)}_{a}(\Omega,w)} \nearrow \|f\|_{L^{q)}_{a}(\Omega,w)}$$;$

		\item[$($iv$)$] if $E\subset \Omega$ is a measurable bounded set and  $a\in L^q(\Omega,w)$, then $\chi_E \in {L^{q)}_{a} (\Omega,w)}$$;$

		\item[$($v$)$] if $w \in A_q(\Omega)$ and $a^\delta \in A_q(\Omega)$ for some $\delta > 0$,
		then $L^{q)}_{a} (\Omega,w) \subset L^1_{\loc}(\Omega)$,
		and for a measurable bounded set $E\subset \Omega$, it holds
		\begin{equation}\label{v}
			\int_E |f(x)| \,dx \leq C \|f\|_{L^{q)}_{a}(\Omega,w)},
		\end{equation}
		where a constant $C$ does not depend on $f$.
	\end{itemize}
\end{proposition}
\begin{proof}
	It is easy to check properties (i)--(iii).
	
	Consider a measurable bounded set $E\subset \Omega$,
	for $0<\varepsilon <q-1$, using H\"{o}lder's inequality with the exponents $\frac{q}{q-\varepsilon}$  and $\frac{q}{\varepsilon}$, we get 
	\begin{equation}\label{neq:chiE}
	\begin{aligned} 
		\|\chi_E&\|_{L_a^{q)}(\Omega,w)} = \displaystyle\sup_{0<\varepsilon <q-1}  \varepsilon^{\frac{1}{q-\varepsilon}} \bigg( \int\limits_\Omega |\chi_E(x)|^{q-\varepsilon} w(x) a^{\varepsilon}(x)dx \bigg)^{\frac{1}{q-\varepsilon}} \\
		& \le \displaystyle\sup_{0<\varepsilon <q-1}  \varepsilon^{\frac{1}{q-\varepsilon}} \bigg( \int\limits_\Omega |\chi_E(x)|^q w(x) dx\bigg)^{\frac{1}{q}}\bigg( \int\limits_\Omega a^q(x) w(x) dx\bigg)^{\frac{\varepsilon}{q(q-\varepsilon)}} \\
		& = \|\chi_E\|_{L^q(\Omega,w)} \displaystyle\sup_{0<\varepsilon <q-1}  \varepsilon^{\frac{1}{q-\varepsilon}}\bigg( \int\limits_\Omega a^{q-\varepsilon}(x)  w(x) a^\varepsilon(x) dx\bigg)^{\frac{1}{q-\varepsilon}} \|a\|^{-1}_{L^q(\Omega,w)} \\
		&  = \|\chi_E\|_{L^q(\Omega,w)} \,  \|a\|_{L^{q)}_{a}(\Omega,w)} \, \|a\|^{-1}_{L^q(\Omega,w).} 
	\end{aligned}
	\end{equation}
	Then the point (iv) follows from \eqref{12} and \eqref{neq:chiE}. 

	To prove (v) let us consider a compact $K\subset \Omega$ and $f \in L^{q)}_{a} (\Omega,w)$. By Lemma~\ref{lemma5} we know that there exists $0<\varepsilon_0<\delta$ such that $wa^{\varepsilon_0} \in A_{q-\varepsilon_0}(\Omega)$.
	Then by the H{\"o}lder inequality, we obtain
	\begin{equation}\label{est:loc}
	\begin{aligned}
		\int\limits_{K} |f(x)|\, dx 
		& = \int\limits_{K} |f(x)| (w a^{\varepsilon_0})^{\frac{1}{q-\varepsilon_0}} (w a^{\varepsilon_0})^{-\frac{1}{q-\varepsilon_0}}\, dx \\
		& \leq C_K |K|\varepsilon_0^\frac{-1}{q-\varepsilon_0} \varepsilon_0^\frac1{q-\varepsilon_0}\|f\|_{L^{q-\varepsilon_0}(\Omega, w a^{\varepsilon_0})} [w a^{\varepsilon_0}]_{A_{q-\varepsilon_0}(\Omega)}^{\frac{1}{q-\varepsilon_0}} \left(\int\limits_K wa^{\varepsilon_0}\right)^{-\frac{1}{q-\varepsilon_0}}\\
		& \leq C_K |K|\|f\|_{L^{q)}_a(\Omega, w)}  \left(\varepsilon_0\int\limits_K wa^{\varepsilon_0}\right)^{-\frac{1}{q-\varepsilon_0}} [w a^{\varepsilon_0}]_{A_{q-\varepsilon_0}(\Omega)}^{\frac{1}{q-\varepsilon_0}} < \infty.
	\end{aligned}
	\end{equation}
	The estimate \eqref{v} can be obtained similarly by corresponding replacement of $K$ by $E$ in \eqref{est:loc}.
\end{proof}
\begin{remark}\label{rem:notBFS}
By Proposition~\ref{prop:generalizedGL} it is easy to see that $ L^{q)}(\Omega)$, with $\Omega$ bounded, is a Banach function space.
The space $L^{q)}_{a} (\Omega,w)$ may instead not be a Banach function space 
for an arbitrary choice of weight functions $a$ and $w$.
\end{remark}

The following theorem, proved in \cite{Uma2014}, gives the sufficient condition for the maximal operator $M\colon L^{q)}_{a}(\mathbb R^n,w) \to L^{q)}_{a}(\mathbb R^n,w)$ to be bounded. 
 
\begin{theorem} \label{thm:E}
	Let $1 < q < \infty$ and $\delta>0$ be such that $a^\delta \in A_q(\mathbb R^n ) $ for some $a \in L^q(\mathbb R^n,w)$. 
	If $w\in A_q(\mathbb R^n)$, then the maximal operator is bounded on $L^{q)}_{a}(\mathbb R^n,w)$, i.e. 
	\begin{equation}\label{lhe}
	\|Mg\|_{L^{q)}_{a}(\mathbb R^n,w)} \leq K_q \|g\|_{L^{q)}_{a}(\mathbb R^n,w)}
	\end{equation}
	for all $g \in L^{q)}_{a}(\mathbb R^n,w)$, where $K_q$ is a positive constant.
\end{theorem}

\begin{remark}
	It is unknown if the condition $w\in A_q(\mathbb R^n)$ is also necessary for the boundedness of $M$ on $L^{q)}_{a}(\mathbb R^n,w)$. 
\end{remark}

As a consequence of Theorem~\ref{thm:E} we obtain the boundedness of the maximal operator on grand Lebesgue spaces in the case of a bounded domain $\Omega\subset\mathbb R^n$.

\begin{proposition} \label{HLGL}
Let $1<q<\infty$,  $\Omega\subset\mathbb R^n$, be a bounded domain and  $t \in (0,\infty]$ be a fixed number. 
Then the inequality
\begin{equation}\label{HLint0}
\|M_tg\|_{L^{q)}(\Omega)}\leq K_q
\|g\|_{L^{q)}(\Omega)}
\end{equation}
holds for all  $g\in L^{q)}(\Omega)$ with some constant $K_q$ independent of $g$. 
\end{proposition}

\begin{proof}
We take a weight 
$a\equiv1$ in $\mathbb R^n$ and a weight $w\colon \mathbb R^n\to(0,\infty)$
such that $w\in A_q(\mathbb R^n)\cap L^1(\mathbb R^n)$ and $w\equiv1$ in some cube $Q\supset\Omega$ big enough such that $B(x,t)\subset Q$ for any 
$x\in\Omega$. 
Then we have
\begin{equation*}
wa^\varepsilon=w\in L^{1}_{\loc}(\Omega) \quad\text{for all $\varepsilon
\in (0, q -1)$.}
\end{equation*}
We extend $g\in L^{q)}(\Omega)$ by 0 outside $\Omega$.
By \eqref{lhe} we get
\begin{equation*}\label{HLint1}
	\|M_tg\|_{L^{q)}(\Omega)}\leq\|M_tg\|_{L^{q)}_1(\mathbb R^n,w)}\leq K_q \|g\|_{L^{q)}_{a}(\mathbb R^n,w)}=
	K_q \|g\|_{L^{q)}(\Omega)}
\end{equation*}
and \eqref{HLint0} is proved.
\end{proof}

\subsection{Grand Sobolev spaces}\label{sec:grand_sobolev}
\setcounter{equation}{0}

The grand version of Sobolev spaces $W^{1,q)}(\Omega)$, if $\Omega$ is a bounded domain, was defined and studied in \cite{Sbo1996}.
We begin with the following definition.

\begin{definition}
Let $1 < q <\infty$, $\Omega \subset\mathbb R^n$ be a bounded domain. 
The {\it grand Sobolev space}, denoted by $W^{1,q)}(\Omega)$, consists of all  $f\in \bigcap\limits_{0<\varepsilon<q-1} W^{1,q-\varepsilon}(\Omega)$ such that
\begin{equation}\label{gss}
\|f\|_{W^{1,q)}(\Omega)} :=
\sup\limits_{0<\varepsilon< q-1}
\varepsilon^\frac1{q-\varepsilon}
\|f\|_{W^{1,q-\varepsilon}(\Omega)}<\infty.
\end{equation}
\end{definition}

\noindent In \cite{Sbo1996} the grand Sobolev space was defined 
by the means of the norm
\begin{equation}\label{gss1}
	|||f|||_{W^{1,q)}(\Omega)}: =
	\|f\|_{L^{q)}(\Omega)}
	+\|\nabla f\|_{L^{q)}(\Omega) <\infty.}
\end{equation}

\noindent However, the norms \eqref{gss} and \eqref{gss1} are equivalent.

\begin{proposition}\label{kl} The following estimates hold
\begin{equation*}
	\frac{1}{4}|||f|||_{W^{1,q)}(\Omega)}\le  \|f\|_{W^{1,q)}(\Omega)}\le |||f|||_{W^{1,q)}(\Omega)}.
\end{equation*}
\end{proposition}
\begin{proof} We use H\"{o}lder's inequality for sums and obtain
\begin{equation*}
\begin{aligned}
	|||f|||_{W^{1,q)}(\Omega)}&=  \sup\limits_{0<\varepsilon<q-1 } \varepsilon^\frac{1}{q-\varepsilon}\|f\|_{L^{q-\varepsilon}(\Omega)}+\sup\limits_{0<\varepsilon<q-1 } \varepsilon^\frac{1}{q-\varepsilon}\|\nabla f\|_{L^{q-\varepsilon}(\Omega)} 
	\\ &
	\le 2 \sup\limits_{0<\varepsilon<q-1 } \varepsilon^\frac{1}{q-\varepsilon}\left(\|f\|_{L^{q-\varepsilon}(\Omega)}+ \|\nabla f\|_{L^{q-\varepsilon}(\Omega)}\right)
	\\ &
	\le  2 \sup\limits_{0<\varepsilon<q-1 } 2^{\frac{q-\varepsilon-1}{q-\varepsilon}} \varepsilon^\frac{1}{q-\varepsilon}\left(\|f\|_{L^{q-\varepsilon}(\Omega)}^{q-\varepsilon}+ \|\nabla f\|_{L^{q-\varepsilon}(\Omega)}^{q-\varepsilon}\right)^{\frac{1}{q-\varepsilon}}
	\\ &
	\le 2^2 \|f\|_{W^{1,q)}(\Omega).}
\end{aligned}
\end{equation*}
Also, we have
\begin{equation*}\label{35}
\begin{aligned}
	 \|f\|_{W^{1,q)}(\Omega)}&
	=  \sup\limits_{0<\varepsilon<q-1 } \varepsilon^\frac{1}{q-\varepsilon}\bigg(\|f\|_{L^{q-\varepsilon}(\Omega)}^{q-\varepsilon}+ \|\nabla f\|_{L^{q-\varepsilon}(\Omega)}^{q-\varepsilon}\bigg) ^{\frac{1}{q-\varepsilon}}
	\\ &
	\le\sup\limits_{0<\varepsilon<q-1 } \varepsilon^\frac{1}{q-\varepsilon}\bigg(\|f\|_{L^{q-\varepsilon}(\Omega)}+ \|\nabla f\|_{L^{q-\varepsilon}(\Omega)}\bigg)
	\\ &
	\le \sup\limits_{0<\varepsilon<q-1 } \varepsilon^\frac{1}{q-\varepsilon}\|f\|_{L^{q-\varepsilon}(\Omega)}+  \sup\limits_{0<\varepsilon<q-1 } \varepsilon^\frac{1}{q-\varepsilon}\|\nabla f\|_{L^{q-\varepsilon}(\Omega)}
	\\&
	=|||f|||_{W^{1,q)}(\Omega)}
\end{aligned}
\end{equation*}
and the assertion follows.
\end{proof}

Now, we define the generalized grand Sobolev space on  $\Omega\subset\mathbb R^n$, $|\Omega| \le \infty$, as follows.
\begin{definition}\label{def:generalizedGS}
	Let $1<q<\infty$ and $w$, $a$ be weight functions on $\Omega$ such that
	$ w a^\varepsilon \in L^{1}_{\loc} (\Omega)$ for all $\varepsilon \in (0, q-1)$.
	The \textit{generalized grand Sobolev space} $W^{1,q)}_{a}(\Omega,w)$ is defined as the collection of all  $f \in  L^{q)}_{a}(\Omega, w)$ having a weak gradient
	$\nabla f$ in  $L^{q)}_{a}(\Omega, w)$, 
	equipped with the norm
	\begin{equation*}\label{kl1}
		\|f\|_{W^{1,q)}_{a}(\Omega,w)} :=
		\sup\limits_{0<\varepsilon< q-1}
		\varepsilon^\frac1{q-\varepsilon}\|f\|_{{W^{1,q-\varepsilon}(\Omega,w a^\varepsilon)}}<\infty.
	\end{equation*}
\end{definition}

\begin{remark}
By the definition of the space $W^{1,q)}_{a}(\Omega,w)$ and Proposition~\ref{kl}, it can be shown that 
\begin{equation}\label{kl2}
\frac{1}{4}\left(\|f\|_{L^{p)}_{a}(\Omega,w)}+ \|\nabla f\|_{L^{p)}_{a}(\Omega,w)}\right) \leq \|f\|_{W^{1,q)}_{a}(\Omega,w)}\le \|f\|_{L^{p)}_{a}(\Omega,w)}+ \|\nabla f\|_{L^{p)}_{a}(\Omega,w)}.
\end{equation}
\end{remark}

\begin{remark}
	If the weight 
	$w\in A_q(\Omega)$ 
	then the weighted Sobolev space 
	$W^{1,q}(\Omega,w)$
	is a Banach space and functions 
	$f\in W^{1,q}(\Omega,w)$
	belong also to 
	$W^{1,1}_{\loc}(\Omega)$,
	see
	\cite[Proposition 2.1]{Kil1994}.
	Similar properties are valid for grand Sobolev spaces
	$W^{1,q)}_{a}(\Omega,w)$. 
\end{remark}

\begin{proposition}\label{prop:W11_loc}
	Let $1<q<\infty$, $wa^{\varepsilon}\in L^1_{\loc}{(\Omega)}$ for all $\varepsilon \in (0, q-1)$, and $a\in L^q{(\Omega,w)}$. 
	Let also $w\in A_q(\Omega)$ and $a^{\delta}\in A_q(\Omega)$ for some $\delta>0$.
	Then $W^{1,q)}_{a}(\Omega,w) \subset W^{1,1}_{\loc}(\Omega)$, 
	and a Cauchy sequence $\{f_n\}_{n\in\mathbb{N}}\subset W^{1,q)}_{a}(\Omega,w)$
	is also Cauchy in $W^{1,1}(D)$ for any $D\Subset \Omega$.
\end{proposition}
\begin{proof}

Let $D\Subset \Omega$. From the proof of Proposition~\ref{prop:generalizedGL}(v)
we obtain
\begin{equation*}
\begin{aligned}
	\int\limits_{D} |f(x)| + |\nabla f(x)|\, dx 
	\leq 4 C_D |D|\left(\varepsilon_0\int\limits_D wa^{\varepsilon_0}\right)^{-\frac{1}{q-\varepsilon_0}} [w a^{\varepsilon_0}]_{A_{q-\varepsilon_0}(\Omega)}^{\frac{1}{q-\varepsilon_0}}  \|f\|_{W^{q)}_a(\Omega, w)}.
\end{aligned}
\end{equation*}
Hence, 
$f \in W^{1,1}_{\loc}(\Omega)$.
Moreover, for any Cauchy sequence $\{f_n\}_{n\in\mathbb{N}}\subset W^{1,q)}_{a}(\Omega,w)$
we have 
$\|f_n - f_m\|_{W^{1,1}(D)}\leq C\|f_n - f_m\|_{W^{1,q)}_a(\Omega,w)}$,
i.e.\ $\{f_n\}_{n\in\mathbb{N}}$ is also a Cauchy sequence in $W^{1,1}(D)$.
\end{proof}

\begin{proposition}\label{prop:banach}
Let $1<q<\infty$, $wa^{\varepsilon}\in L^1_{\loc}{(\Omega)}$ for all $\varepsilon \in (0, q-1)$, and $a\in L^q{(\Omega,w)}$. 
Then the grand Sobolev space $W^{1,q)}_{a}(\Omega,w)$ is a Banach space if $w\in A_q(\Omega)$ and $a^{\delta}\in A_q(\Omega)$ for some $\delta>0$.
\end{proposition}

\begin{proof} It is enough to prove completeness. 
Let $\{f_n\}_{n\in\mathbb{N}}$ be a Cauchy sequence in $W^{1,q)}_{a}(\Omega,w)$,
by Proposition~\ref{prop:W11_loc} it is so in $W^{1,1}_{\loc}(\Omega)$, denote the limit by $f$. 
On the other hand, $\{f_n\}_{n\in\mathbb{N}}$ and $\{\partial_{\alpha} f_n\}_{n\in\mathbb{N}}$ are Cauchy in Banach spaces $L^{q)}_a(\Omega,w)$ for all multi-indexes $\alpha$, $|\alpha|=1$. 
Let $h$ and $h_{\alpha}$ be corresponding limits.
Then it is easy to see that 
$h=f$, $h_{\alpha} = \partial_\alpha f$ is a weak derivative of $f$,
and $\|f_n - f\|_{W^{1,q)}_{a}(\Omega,w)} \to 0$, i.e.\ $W^{1,q)}_{a}(\Omega,w)$ is complete.
\end{proof}

\begin {theorem}\label{th2}
The embedding  
$$W^{1,q}(\Omega,w) \hookrightarrow  W^{1,q)}_{a}(\Omega,w)$$ 
holds if $a\in L^q{(\Omega,w)}$. 
\end{theorem}

\begin{proof} Let $a \in L^q(\Omega,w) $ and $f \in W^{1,q}(\Omega,w)$. 
Similarly to \eqref{neq:chiE} we can obtain
\begin{equation} \label{eq01} 
	\|f\|_{L_a^{q)}(\Omega,w)} \leq \|f\|_{L^q(\Omega,w)} \,  \|a\|_{L^{q)}_{a}(\Omega,w)} \, \|a\|^{-1}_{L^q(\Omega,w)} 
\end{equation}
\noindent and
\begin{equation}\label{55}
	\|\nabla f\|_{L_a^{q)}(\Omega,w)}\le\|\nabla f\|_{L^q(\Omega,w)} \,  \|a\|_{L^{q)}_{a}(\Omega,w)} \, \|a\|^{-1}_{L^q(\Omega,w).}
\end{equation}
Since \eqref{kl2} holds, using \eqref{eq01} and \eqref{55}, we obtain
\begin{equation*}\label{56}
\begin{aligned}
\|f\|_{W^{1,q)}_{a}(\Omega,w)} & 
\le   \|a\|_{L^{q)}(\Omega,w)} \, \|a\|^{-1}_{L^q(\Omega,w)} \left(\|f\|_{L^{q}(\Omega,w)}+ \|\nabla f\|_{L^{q}(\Omega,w)}\right) \\
& \le  4 \|a\|_{L^{q)}(\Omega,w)} \, \|a\|^{-1}_{L^q(\Omega,~w)}\|f\|_{W^{1,q}(\Omega,w)}\\
& = K_a  \|f\|_{W^{1,q}(\Omega,w),}
\end{aligned}
\end{equation*}
where $ K_a = 4\|a\|_{L^{q)}(\Omega,w)} \, \|a\|^{-1}_{L^q(\Omega,w)}$. 
Now since $a\in L^q(\Omega ,w)$ and in view of \eqref{12}, the embedding $L^{q}(\Omega,w) \hookrightarrow L^{q)}_{a}(\Omega,w)$ holds, it follows that $K_a< \infty$.
\end{proof}

In view of Theorem \ref{th2} and embeddings \eqref{12}, if $a\in L^q(\Omega,w)$ we have that the following embeddings hold for all $0< \varepsilon <q-1$
\begin{equation*}\label{re}
	W^{1,q}(\Omega,w) \hookrightarrow  W^{1,q)}_{a}(\Omega,w) \hookrightarrow W^{1,q-\varepsilon}(\Omega,w a^\varepsilon).
\end{equation*}

For further discussion of grand Lebesgue and Sobolev spaces we refer the reader to \cite{CapFioKar2008,CarSbo1997,DonSboSch2013,FioKar2004,FioMerRak2001,FioMerRak2002,FioSbo1998,FioForGog2018,ForOstSir2019,Gre1993,GreIwaSbo1997,JaiSinSinSte2019,Kok2010,KokMes2009,Mol2019,Sbo1998} and \cite[\S 7.2]{CasRaf2016}.

In the celebrated paper \cite{Muc1972} it was proved that the maximal operator $M\colon L^q(\mathbb R^n,w) \rightarrow L^q(\mathbb R^n,w)$ is bounded for $1 < q < \infty$ 
if and only if $w \in A_q(\mathbb R^n)$. 
Later, the class $A_q(\mathbb R^n )$ turned out to be very useful since it also characterizes the $L^q$-boundedness of 
several other operators such as the Hilbert transform \cite{HunMucWhe1973}. 
In fact, the same class characterizes also 
the boundedness of the maximal operator on grand Lebesgue spaces
(one-dimensional result was proved in \cite{FioGupJai2008}).

Let us present explicitly the pointwise description for weighted Sobolev spaces, which was not stated before according to the authors' knowledge. 
Some related properties can be found in the proof of \cite[Corollary~3]{Haj1996}.

\begin{corollary}\label{cor:weighted_Sobolev} 
	Let $1<q<\infty$ and $w\in A_q(\mathbb R^n)$.
	A~function  
	$f$ belongs to $W^{1,q}(\RR^n,w)$
	if and only if 
	$f\in L^{q}(\RR^n,w)$ and 
	there exist a non negative function $g\in L^{q}(\RR,w)$ 
	such that the inequality
	\begin{equation}\label{description}
	|f(x)-f(y)|\leq |x-y|
	(g(x)+g(y))
	\end{equation}
	holds for $x$, $y \in \mathbb R^n$ a.e.
\end{corollary}

Due to Proposition~\ref{HLGL}, the following result is the corollary of Theorem~\ref{thm:D} for grand Sobolev spaces, defined on bounded domains.
\begin{corollary}\label{40} 
	Let $1<q<\infty$ and  $\Omega\subset\mathbb
	R^n$ be a bounded domain.
	A~function  
	$f$ belongs to $W^{1,q)}(\Omega)$
	if and only if 
	$f\in L^{q)}(\Omega)$ and 
	there exist a non negative function $g\in L^{q)}(\Omega)$ 
	such that the inequality
	\eqref{description}
	holds for all points~
	$x$, $y\in \Omega\setminus S$ with $B(x,3|x-y|)\subset\Omega$,
	where $S\subset\Omega$ is a set of measure zero.
\end{corollary}

\begin{remark}
If $f \in W^{1,q)}(\Omega) $ then $|f|\in W^{1,q)}(\Omega) $ because by Corollary~\ref{40}, there exist a non negative function $g\in L^{q)}(\Omega)$ and a set $S\subset \Omega$ of measure zero such that 
$$\Big||f(x)|-|f(y)|\Big|\leq|f(x)-f(y)|\leq |x-y|
(g(x)+g(y))$$
holds for all points~
$x$, $y\in \Omega\setminus S$ with $B(x,3|x-y|)\subset\Omega$.
\end{remark}

And finally, we formulate the pointwise characterization for generalized grand Sobolev spaces.

\begin{corollary}\label{cor:GS} 
Let $1<q<\infty$ and $\delta>0$ be such that $a^\delta\in A_q(\mathbb R^n)$ for some $a\in L^q(\mathbb R^n,w)$, and $w\in A_q(\mathbb R^n)$. 
Then  $f \in W^{1,q)}_{a}(\mathbb R^n,w)$ if and only if $f\in L^{q)}_{a}(\mathbb R^n,w)$ and  there exists a non negative function $g\in L^{q)}_{a}(\mathbb R^n,w)$ such that   the inequality
\eqref{description}
holds for $x$, $y \in \mathbb R^n$ a.e.
\end{corollary}

\section*{Acknowledgement}
The authors thank the anonymous reviewer for critically reading and comments, which helped
improve and clarify this manuscript.

PJ acknowledges the MATRICS Research Grant No.~MTR/2017/000126 of SERB, Department of Science and Technology, India.
AM was partially funded by the Austrian Science Fund (FWF) through the project M~2670.
AM and SV were supported by the Mathematical Center in Akademgorodok under agreement No.~075-15-2019-1613 with the Ministry of Science and Higher Education of the Russian Federation.


\newpage
\bibliographystyle{plain}
\bibliography{biblio_PMV}

\begin{thebibliography}{10}

\bibitem{AlbCiaSbo2017}
A.~Alberico, A.~Cianchi, and C.~Sbordone.
\newblock Continuity properties of solutions to the {$p$}-{L}aplace system.
\newblock {\em Adv. Calc. Var.}, 10(1):1--24, 2017.

\bibitem{AnaFio2015}
G.~Anatriello and A.~Fiorenza.
\newblock Fully measurable grand {L}ebesgue spaces.
\newblock {\em J. Math. Anal. Appl.}, 422(2):783--797, 2015.

\bibitem{Bal1977}
J.~M. Ball.
\newblock Convexity conditions and existence theorems in nonlinear elasticity.
\newblock {\em Arch. Rational Mech. Anal.}, 63(4):337--403, 1977.

\bibitem{BenSha1988}
C.~Bennett and R.~Sharpley.
\newblock {\em Interpolation of operators}, volume 129 of {\em Pure and Applied
  Mathematics}.
\newblock Academic Press, Inc., Boston, MA, 1988.

\bibitem{Boj1991}
B.~Bojarski.
\newblock Remarks on some geometric properties of {S}obolev mappings.
\newblock In {\em Functional analysis \& related topics ({S}apporo, 1990)},
  pages 65--76. World Sci. Publ., River Edge, NJ, 1991.

\bibitem{BojHaj1993}
B.~Bojarski and P.~Haj{\l}asz.
\newblock Pointwise inequalities for {S}obolev functions and some applications.
\newblock {\em Studia Math.}, 106(1):77--92, 1993.

\bibitem{BonCapHajShaTys2020}
M.~Bonk, L.~Capogna, P.~Haj{\l}asz, N.~Shanmugalingam, and J.~Tyson.
\newblock Analysis in metric spaces.
\newblock {\em Notices Amer. Math. Soc.}, 67(2):253--256, 2020.

\bibitem{BreGal1980}
H.~Br\'{e}zis and T.~Gallouet.
\newblock Nonlinear {S}chr\"{o}dinger evolution equations.
\newblock {\em Nonlinear Anal.}, 4(4):677--681, 1980.

\bibitem{CapFioKar2008}
C.~Capone, A.~Fiorenza, and G.~E. Karadzhov.
\newblock Grand {O}rlicz spaces and global integrability of the {J}acobian.
\newblock {\em Math. Scand.}, 102:131--148, 2008.

\bibitem{CapForGio2013}
C.~Capone, M.~R. Formica, and R.~Giova.
\newblock Grand {L}ebesgue spaces with respect to measurable functions.
\newblock {\em Nonlinear Anal.}, 85:125--131, 2013.

\bibitem{CarSbo1997}
M.~Carozza and C.~Sbordone.
\newblock The distance to {$L^{\infty}$} in some function spaces and
  applications.
\newblock {\em Differ. Integral Equ. Appl.}, 10:599--607, 1997.

\bibitem{CasRaf2016}
R.~E. Castillo and H.~Rafeiro.
\newblock {\em An introductory course in {L}ebesgue spaces}.
\newblock CMS Books in Mathematics/Ouvrages de Math\'{e}matiques de la SMC.
  Springer, [Cham], 2016.

\bibitem{CoiFef1974}
R.~R. Coifman and C.~Fefferman.
\newblock Weighted norm inequalities for maximal functions and singular
  integrals.
\newblock {\em Studia Math.}, 51:241--250, 1974.

\bibitem{DieHarHasRuz2017}
L.~Diening, P.~Harjulehto, P.~H\"{a}st\"{o}, and M.~Ru\v{z}i\v{c}ka.
\newblock {\em Lebesgue and {S}obolev spaces with variable exponents}, volume
  2017 of {\em Lecture Notes in Mathematics}.
\newblock Springer, Heidelberg, 2011.

\bibitem{DieRuz2003}
L.~Diening and M.~Ru\v{z}i\v{c}ka.
\newblock Calder\'{o}n-{Z}ygmund operators on generalized {L}ebesgue spaces
  {$L^{p(\cdot)}$} and problems related to fluid dynamics.
\newblock {\em J. Reine Angew. Math.}, 563:197--220, 2003.

\bibitem{DonSboSch2013}
L.~D'Onofrio, C.~Sbordone, and R.~Schiattarella.
\newblock {G}rand {S}obolev spaces and their applications in geometric function
  theory and {PDE}s.
\newblock {\em J. Fixed Point Theory Appl.}, 13:309--340, 2013.

\bibitem{EvaGar1992}
L.~C. Evans and R.~F. Gariepy.
\newblock {\em Measure theory and fine properties of functions}.
\newblock Studies in Advanced Mathematics. CRC Press, Boca Raton, FL, 1992.

\bibitem{Fio2000}
A.~Fiorenza.
\newblock Duality and reflexivity in grand {L}ebesgue spaces.
\newblock {\em Collect. Math.}, 51:131--148, 2000.

\bibitem{FioForGog2018}
A.~Fiorenza, M.~R. Formica, and A.~Gogatishvili.
\newblock On grand and small {L}ebesgue and {S}obolev spaces and some
  applications to {PDE}'s.
\newblock {\em Differ. Equ. Appl.}, 10(1):21--46, 2018.

\bibitem{FioForRosSou2019}
A.~Fiorenza, M.~R. Formica, T.~Roskovec, and F.~Soudsk\'{y}.
\newblock Gagliardo-{N}irenberg inequality for rearrangement-invariant {B}anach
  function spaces.
\newblock {\em Atti Accad. Naz. Lincei Rend. Lincei Mat. Appl.},
  30(4):847--864, 2019.

\bibitem{FioGupJai2008}
A.~Fiorenza, B.~Gupta, and P.~Jain.
\newblock The maximal theorem for weighted grand {L}ebesgue spaces.
\newblock {\em Studia Math.}, 188(2):123--133, 2008.

\bibitem{FioKar2004}
A.~Fiorenza and G.~E. Karadzhov.
\newblock Grand and small {L}ebesgue spaces and their analogs.
\newblock {\em J. Anal. Appl.}, 23:657--681, 2004.

\bibitem{FioKrb1998}
A.~Fiorenza and M.~Krbec.
\newblock A formula for the {B}oyd indices in {O}rlicz spaces.
\newblock {\em Funct. {A}pprox. {C}omment. {M}ath.}, 26:173--179, 1998.

\bibitem{FioMerRak2001}
A.~Fiorenza, A.~Mercaldo, and J.~M. Rakotoson.
\newblock Regularity and comparison results in grand {S}obolev spaces for
  parabolic equations with measure data.
\newblock {\em Appl. Math. Lett.}, 14:979--981, 2001.

\bibitem{FioMerRak2002}
A.~Fiorenza, A.~Mercaldo, and J.~M. Rakotoson.
\newblock Regularity and uniqueness results in grand {S}obolev spaces for
  parabolic equations with measure data.
\newblock {\em Discrete Contin. Dyn. Syst.}, 8(4):893--906, 2002.

\bibitem{FioSbo1998}
A.~Fiorenza and C.~Sbordone.
\newblock Rexistence and uniqueness results for solutions of nonlinear
  equations with right hand side in {$L^1$}.
\newblock {\em Studia Math.}, 127(3):223--231, 1998.

\bibitem{ForOstSir2019}
M.~R. {Formica}, E.~{Ostrovsky}, and L.~{Sirota}.
\newblock {Grand Lebesgue spaces are really Banach algebras relative to the
  convolution on unimodular locally compact groups}.
\newblock {\em arXiv e-prints}, page arXiv:1904.09226, April 2019.

\bibitem{GilTru1983}
D.~Gilbarg and N.~S. Trudinger.
\newblock {\em Elliptic partial differential equations of second order}, volume
  224 of {\em Grundlehren der Mathematischen Wissenschaften [Fundamental
  Principles of Mathematical Sciences]}.
\newblock Springer-Verlag, Berlin, second edition, 1983.

\bibitem{Gre1993}
L.~Greco.
\newblock A remark on the equality {$\det Df = \operatorname{Det} Df$}.
\newblock {\em Differential Integral Equations}, 6:1089--1100, 1993.

\bibitem{GreIwaSbo1997}
L.~Greco, T.~Iwaniec, and C.~Sbordone.
\newblock Inverting the {$p$}-harmonic operator.
\newblock {\em Manuscripta Math.}, 92:249--258, 1997.

\bibitem{Haj1996}
P.~Haj{\l}asz.
\newblock Sobolev spaces on an arbitrary metric space.
\newblock {\em Potential Anal.}, 5(4):403--415, 1996.

\bibitem{Haj2003}
P.~Haj{\l}asz.
\newblock A new characterization of the {S}obolev space.
\newblock {\em Studia Math.}, 159(2):263--275, 2003.

\bibitem{Hed1972}
L.~I. Hedberg.
\newblock On certain convolution inequalities.
\newblock {\em Proc. Amer. Math. Soc.}, 36:505--510, 1972.

\bibitem{HeiKosShaTys2015}
J.~Heinonen, P.~Koskela, N.~Shanmugalingam, and J.~T. Tyson.
\newblock {\em Sobolev Spaces on Metric Measure Spaces: An Approach Based on
  Upper Gradients}.
\newblock New Mathematical Monographs. Cambridge University Press, 2015.

\bibitem{HunMucWhe1973}
R.~Hunt, B.~Muckenhoupt, and R.~Wheeden.
\newblock Weighted norm inequalities for the conjugate function and {H}ilbert
  transform.
\newblock {\em Trans. Amer. Math. Soc.}, 176:227--251, 1973.

\bibitem{IwaSbo1992}
T.~Iwaniec and C.~Sbordone.
\newblock On the integrability of the {J}acobian under minimal hypotheses.
\newblock {\em Arch. Ration. Mech. Anal.}, 119:129--143, 1992.

\bibitem{JaiSinSinSte2019}
P.~Jain, A.~P. Singh, M.~Singh, and V.~D. Stepanov.
\newblock Sawyer's duality principle for grand {L}ebesgue spaces.
\newblock {\em Math. Nachr.}, 292(4):841--849, 2019.

\bibitem{JaiSinSin2017-2}
P.~Jain, M.~Singh, and A.~P. Singh.
\newblock Duality of fully measurable grand {L}ebesgue space.
\newblock {\em Trans. A. Razmadze Math. Inst.}, 171(1):32--47, 2017.

\bibitem{JaiSinSin2017-1}
P.~Jain, M.~Singh, and A.~P. Singh.
\newblock Recent trends in grand {L}ebesgue spaces.
\newblock In {\em Function Spaces and Inequalities}, volume 206 of {\em
  Springer Proceedings in Mathematics \& Statistics}, pages 137--159. Springer,
  2017.

\bibitem{Kil1994}
T.~Kilpel\"{a}inen.
\newblock Weighted {S}obolev spaces and capacity.
\newblock {\em Ann. Acad. Sci. Fenn. Ser. A I Math.}, 19(1):95--113, 1994.

\bibitem{Kit1997}
H.~Kita.
\newblock On {H}ardy--{L}ittlewood maximal functions in {O}rlicz spaces.
\newblock {\em Math. Nachr.}, 183:135--155, 1997.

\bibitem{Kok2010}
V.~Kokilashvili.
\newblock Boundedness criterion for singular integrals in weighted grand
  {L}ebesgue spaces.
\newblock {\em J. Math. Sci.}, 170:20--33, 2010.

\bibitem{KokMes2009}
V.~Kokilashvili and A.~Meskhi.
\newblock A note on the boundedness of the {H}ilbert transform in weighted
  grand {L}ebesgue spaces.
\newblock {\em Georgian Math. J.}, 16:547--551, 2009.

\bibitem{KosSak2008}
P.~Koskela and E.~Saksman.
\newblock Pointwise characterizations of {H}ardy--{S}obolev functions.
\newblock {\em Math. Res. Lett.}, 15(4):727--744, 2008.

\bibitem{Mol2019}
A.~Molchanova.
\newblock A note on the continuity of minors in grand {L}ebesgue spaces.
\newblock {\em J. Fixed Point Theory Appl.}, 21(2):Art. 49, 13, 2019.

\bibitem{MolRosSou2019}
A.~Molchanova, T.~Roskovec, and F.~Soudsk\'y.
\newblock Regularity of the inverse mapping in {B}anach function spaces, 2019.

\bibitem{Muc1972}
B.~Muckenhoupt.
\newblock Weighted norm inequalities for the {H}ardy maximal function.
\newblock {\em Trans. Amer. Math. Soc.}, 165:207--226, 1972.

\bibitem{Mus2019}
V.~Musil.
\newblock Fractional maximal operator in {O}rlicz spaces.
\newblock {\em J. Math. Anal. Appl.}, 474(1):94--115, 2019.

\bibitem{PicKufJohFuc2013}
L.~Pick, A.~Kufner, O.~John, and S.~Fu\v{c}\'{i}k.
\newblock {\em Function spaces. {V}ol. 1}, volume~14 of {\em de Gruyter Ser.
  Nonlinear Anal. Appl.}
\newblock Walter de Gruyter \& Co., Berlin, 2013.

\bibitem{SamUma2011}
S.~G. Samko and S.~M. Umarkhadzhiev.
\newblock On {I}waniec--{S}bordone spaces on sets which may have infinite
  measure.
\newblock {\em Azerb. J. Math.}, 1(1):67--84, 2011.

\bibitem{Sbo1996}
C.~Sbordone.
\newblock Grand {S}obolev spaces and their applications to variational
  problems.
\newblock {\em Matematiche}, 52(2):335--347, 1996.

\bibitem{Sbo1998}
C.~Sbordone.
\newblock Nonlinear elliptic equations with right hand side in nonstandard
  spaces.
\newblock {\em Atti Sem. Mat. Fis. Univ. Modena}, 46 (Suppl.):361--368, 1998.

\bibitem{Sch2005}
R.~L. Schilling.
\newblock {\em Measures, integrals and martingales}.
\newblock Cambridge University Press, New York, 2005.

\bibitem{Tuo2007}
H.~Tuominen.
\newblock Characterization of {O}rlicz--{S}obolev space.
\newblock {\em Ark. Mat.}, 45(1):123--139, 2007.

\bibitem{Uma2014}
S.~M. Umarkhadzhiev.
\newblock Generalization of the notion of grand {L}ebesgue space.
\newblock {\em Russian Math. (Iz. VUZ)}, 58(4):35--43, 2014.

\bibitem{Vod1996}
S.~K. Vodopyanov.
\newblock Monotone functions and quasiconformal mappings on {C}arnot groups.
\newblock {\em Sibirsk. Mat. Zh.}, 37(6):1269--1295, 1996.

\bibitem{Vod2007}
S.~K. Vodopyanov.
\newblock Foundations of the theory of mappings with bounded distortion on
  {C}arnot groups.
\newblock In {\em The interaction of analysis and geometry}, volume 424 of {\em
  Contemp. Math.}, pages 303--344. Amer. Math. Soc., Providence, RI, 2007.

\end{thebibliography}

\end{document}